\documentclass[a4paper,pdftex,reqno,10pt]{amsart}
\usepackage{amssymb,amsmath}
\usepackage{graphicx}        
\usepackage{bm} 
\usepackage{enumerate} 
\usepackage{ytableau} 
\usepackage{url}

\newcommand{\F}{\mathbb{F}}
\newcommand{\Z}{\mathbb{Z}}

\newcommand{\PG}{\operatorname{PG}}
\newcommand{\AG}{\operatorname{AG}}

\newcommand{\vek}[1]{\mathbf{#1}}
\newcommand{\mat}[1]{\mathbf{#1}}

\newcommand{\sdist}{\mathrm{d}_{\mathrm{S}}}

\newcommand{\mindist}{\mathrm{d}}

\newcommand{\imat}{\mathbf{I}}

\newcommand{\points}{\mathcal{P}}

\newcommand{\card}[1]{\##1}
\newcommand{\frobenius}{\mathrm{F}}

\newcommand{\wham}{\mathrm{w}}
\newcommand{\mset}[1]{\mathfrak{#1}}
\newcommand{\PD}{\operatorname{PD}}
\newcommand{\LPD}{\operatorname{LPD}}

\newtheorem{proposition}{Proposition}
\newtheorem{corollary}{Corollary}
\newtheorem{example}{Example}
\newtheorem{theorem}{Theorem}

\begin{document}

\title{Projective Divisible Binary Codes}

\author[Heinlein]{Daniel Heinlein}
\address{Daniel Heinlein, University of Bayreuth, 95440 Bayreuth, Germany, daniel.heinlein@uni-bayreuth.de}

\author[Honold]{Thomas Honold}
\address{Thomas Honold, Zhejiang University, 310027 Hangzhou, China, honold@zju.edu.cn}

\author[Kiermaier]{Michael Kiermaier}
\address{Michael Kiermaier, University of Bayreuth, 95440 Bayreuth, Germany, michael.kiermaier@uni-bayreuth.de}

\author[Kurz]{Sascha Kurz}
\address{Sascha Kurz, University of Bayreuth, 95440 Bayreuth, Germany, sascha.kurz@uni-bayreuth.de}

\author[Wassermann]{Alfred Wassermann}
\address{Alfred Wassermann, University of Bayreuth, 95440 Bayreuth, Germany, alfred.wassermann@uni-bayreuth.de}


\maketitle

\begin{abstract}For which positive integers $n,k,r$ does there exist a
  linear $[n,k]$ code $C$ over $\F_q$ with all codeword weights divisible by
  $q^r$ and such that the columns of a generating matrix of $C$ are
  projectively distinct? The motivation for studying this problem
  comes from the theory of partial spreads, or subspace codes with
  the highest possible minimum distance, 
  since the set of holes of a partial spread of $r$-flats in
  $\PG(v-1,\F_q)$ corresponds to a $q^r$-divisible code with $k\leq v$.
  In this paper we provide an introduction to this problem
  and report on new results for $q=2$.
\end{abstract}

  
\section{Introduction}

Let $q=p^e>1$ be a prime power and $\Delta>1$ an integer.  A linear
code $C$ over $\F_q$ is said to be \emph{$\Delta$-divisible} if the
Hamming weight $\wham(\vek{c})$ of every codeword $\vek{c}\in C$ is
divisible by $\Delta$. The classical examples are self-dual codes over
$\F_2$, $\F_3$ and $\F_4$, which have $\Delta\in\{2,4\}$, $\Delta=3$
and $\Delta=2$, respectively. While self-dual codes or, slightly more
general, $[n,n/2]$ codes cannot have other divisors by the
Gleason-Pierce-Ward Theorem \cite[Ch.~9.1]{huffman-pless03}, there
exist interesting examples in (necessarily) smaller dimension for
every pair $q,\Delta$ in which $\Delta=p^f$ is a power of the
characteristic of $\F_q$. The most well-known example is the family of
$q$-ary $\left[\frac{q^k-1}{q-1},k,q^{k-1}\right]$ simplex codes (dual Hamming
codes), which have constant weight $\Delta=q^{k-1}=p^{e(k-1)}$.  
In the remaining case $\Delta=mp^f$ with $m>1$ and $\gcd(m,p)=1$, a
\emph{$\Delta$-divisible} code is necessarily an $m$-fold replicated
code \cite[Th.~1]{ward2001divisible_survey}, reducing this case to the
former.

Our motivation for studying divisible codes comes from Finite Geometry
and the recently established field of Subspace Coding. A \emph{partial
  $r$-spread} in the projective geometry
$\PG(v-1,\F_q)=\PG(\F_q^v/\F_q)$ is a set of pairwise disjoint
$r$-subspaces of $\F_q^v/\F_q$.\footnote{Here $r$ refers
  to the vector space dimension of
  the subspace (the geometric dimension as a flat of $\PG(v-1,\F_q)$
  being $r-1$), but
  ``disjoint'' means disjoint as point sets in $\PG(v-1,\F_q)$ (the
  corresponding vector space intersection being $\{\vek{0}\}$).} 
To avoid trivialities, we assume $r\geq 2$.

In the case $r\mid v$ the existence of \emph{$r$-spreads}, i.e.,
partial $r$-spreads partitioning the point set of $\PG(v-1,\F_q)$ is
well-known, but in the case $r\nmid v$ (in which spreads cannot exist) the
maximum size of a partial $r$-spread in $\PG(v-1,\F_q)$ is
generally unknown. The problem of determining this maximum size forms
a special case of the so-called \emph{Main Problem of Subspace
  Coding}, which arose from the elegant Koetter-Kschischang-Silva
model for Random Linear Network Coding
\cite{koetter-kschischang08,silva-kschischang-koetter08,silva-kschischang-koetter10}
and is akin to the Main Problem of classical Coding Theory. It asks
for the maximum number of subspaces of $\F_q^v/\F_q$ at mutual
distance $\geq d$ in the subspace metric defined by
$\sdist(X,Y)=\dim(X+Y)-\dim(X\cap Y)$.  If attention is restricted to
subspaces of constant dimension $r$ and $d=2r$ is the maximum
possible distance under this restriction, we recover the original
geometric problem.

We will not discuss the known results about maximal partial spreads in
this paper, for which we refer interested readers to the recent
exhaustive survey \cite{smt:cost}. Instead we will describe
the link between partial spreads and divisible codes
(Section~\ref{sec:link}), formulate a ``Main Problem'' for
projective divisible codes (Section~\ref{sec:main}), discuss some general
divisible code constructions (Section~\ref{sec:const}),
and report on new results for the particular case $q=2$, $\Delta=2^r$
(Section~\ref{sec:q=2}).

\section{Linking Partial Spreads and Divisible
Codes}\label{sec:link}

The link between partial spreads and divisible codes is provided by
the concept of a ``hole'' of a family $\mathcal{S}$ of subspaces of
$\PG(v-1,\F_q)$. A point of $\PG(v-1,\F_q)$ (i.e., a $1$-dimensional
subspace of $\F_q^v/\F_q$) is said to be a \emph{hole} of $\mathcal{S}$ if it
is not covered by (i.e., not incident with) a member of $\mathcal{S}$.
Further, we recall from \cite{dodunekov1998codes,tsfasman-vladut95}
that associating with a linear $[n,k]$-code $C$ the multiset
$\mset{K}_C$ of points generated by the columns of any generating matrix
$\mat{G}$ of $C$ yields a one-to-one correspondence between monomial
equivalence classes of linear $[n,k]$-codes over $\F_q$ without
universal zero coordinate and isomorphism classes of $n$-element
spanning multisets of points in $\PG(k-1,\F_q)$. The relation
$C\mapsto\mset{K}_C$ preserves the
metric in the sense that the weight $\wham(\vek{c})$ of a nonzero codeword
$\vek{c}=\vek{a}\mat{G}$ and the multiplicity $\mset{K}_C(H)
=\sum_{P\in H}\mset{K}_C(P)$ of the corresponding hyperplane
$H=\vek{a}^\perp=\{\vek{x}\in\F_q^k;a_1x_1+\dots+a_kx_k=0\}$ are
related by
$\wham(\vek{a}\mat{G})=n-\mset{K}_C(\vek{a}^\perp)=\mset{K}_C(\points
\setminus\vek{a}^\perp)$,  
where $\points$ denotes the point set of $\PG(k-1,\F_q)$. The code $C$
is $\Delta$-divisible iff the multiset $\mset{K}_C$ is
\emph{$\Delta$-divisible} in the sense that the multiplicity
$\mset{K}_C(A)$ of any ($k-1$)-dimensional affine subspace
$A$ of $\PG(k-1,\F_q)$ is divisible by $\Delta$.

The code $C$ is said to be \emph{projective} if $\mset{K}_C$ is a
set or, equivalently, the $n$ columns of $\mat{G}$ are projectively
distinct. In terms of the minimum distance of the dual code
this can also be expressed as $\mindist(C^\perp)\geq 3$.

\begin{proposition}[compare {\cite[Th.~8]{smt:cost}}]
  \label{prop:link}
  Let $\mathcal{S}$ be a partial $r$-spread in $\PG(v-1,\F_q)$,
  $\mset{H}$ its set of holes, and $C_{\mset{H}}$ any linear $[n,k]$ code over
  $\F_q$ associated with $\mset{H}$ as defined above. Then
  \begin{enumerate}[(i)]
  \item $C_{\mset{H}}$ is projective and $q^{r-1}$-divisible;
  \item the parameters of $C_{\mset{H}}$ satisfy $n=\frac{q^v-1}{q-1}
    -\card{\mathcal{S}}\cdot\frac{q^r-1}{q-1}$ and $k\leq v$.
  \end{enumerate}
\end{proposition}
\begin{proof}
  All assertions except the $q^{r-1}$-divisibility of $C_{\mset{H}}$ are
  straightforward. For the proof of the latter let
  $\mathcal{S}=\{S_1,\dots,S_M\}$, $M=\card{\mathcal{S}}$, and
  consider a generating
  matrix
  \begin{equation*}
    \mat{G}=
    \begin{pmatrix}
      \mat{G}_1&\mat{G}_2&\dots&\mat{G}_M&\mat{H}
    \end{pmatrix}
  \end{equation*}
  of the $q$-ary $\left[\frac{q^v-1}{q-1},v,q^{v-1}\right]$ simplex code,
  partitioned in such a way that the columns of $\mat{G}_j$ account
  for all points in $S_j$ and those of $\mat{H}$ for
  all points in $\mset{H}$. 
  For a nonzero codeword $\vek{c}=\vek{x}\mat{G}$ of the simplex code we have
  \begin{equation*}
    q^{v-1}=\wham(\vek{x}\mat{G})=\sum_{j=1}^M\wham(\vek{x}\mat{G}_j)
    +\wham(\vek{x}\mat{H}).
  \end{equation*}
  Since each matrix $\mat{G}_j$ generates an $r$-dimensional
  simplex code (in the broader sense, i.e., the rows of $\mat{G}_j$
  need not be linearly independent), we have
  $\wham(\vek{x}\mat{G}_j)\in\{0,q^{r-1}\}$. Since $v-1\geq r-1$, it follows that
  $\wham(\vek{x}\mat{H})$ is divisible by $q^{r-1}$ as well. But $\mat{H}$
  generates $C_{\mset{H}}$ and the result follows.
\end{proof}

Proposition~\ref{prop:link} looks rather innocent at the first glance,
but in fact it provides a powerful tool for bounding the size of
partial spreads. This is already illustrated by the following

\begin{corollary}
  \label{cor:beutel}
  If $v\geq 2r+1$ and $v\bmod r=1$, the maximum size of a partial $r$-spread in
  $\PG(v-1,\F_q)$ is
  \begin{equation*}
    \left\lfloor\frac{q^v-1}{q^r-1}\right\rfloor-(q-1)
    =q^{v-r}+q^{v-2r}+\dots+q^{r+1}+1,
  \end{equation*}
  with corresponding number of holes equal to $q^r$.
\end{corollary}
\begin{proof}
  It is readily shown by induction that there exists a partial
  $r$-spread with the required property, the induction step being
  provided by generating matrices in $\F_q^{v\times r}$ of the form
  $\left(
    \begin{smallmatrix}
      \imat_r\\\mat{A}
    \end{smallmatrix}\right)$, where $\mat{A}\in\F_q^{(v-r)\times r}$
  runs through a matrix representation of $\F_{q^{v-r}}$ with the
  last $v-2r$ columns stripped off. The subspaces of the partial
  spread are the column spaces of the matrices $\left(
    \begin{smallmatrix}
      \imat_r\\\mat{A}
    \end{smallmatrix}\right)$, and the anchor of the induction is
  provided by adding the
  column space of $\left(
    \begin{smallmatrix}
      \mat{0}\\\imat_r\end{smallmatrix}\right)\in\F_q^{(2r+1)\times
    r}$ to the $q^{r+1}$ subspaces obtained for $v=2r+1$ in the same
  way as in the inductive step.

  Conversely, let $\mathcal{S}$ by a partial $r$-spread in
  $\PG(v-1,\F_q)$ and $\mset{H}$ its set of holes. By
  Proposition~\ref{prop:link}, the $[n,k]$ code $C_{\mset{H}}$ is
  $q^{r-1}$-divisible satisfying 
  $n=\card{\mset{H}}=1+h\cdot\frac{q^r-1}{q-1}$ 
  for some integer $h$. We must show $h\geq q-1$.

  Assuming $h<q-1$, we have $n=1+h(1+q+\dots+q^{r-1})=n_1+hq^{r-1}$
  with $n_1=1+h(1+q+\dots+q^{r-2})<q^{r-1}$. This implies that the maximum weight
  of $C_{\mset{H}}$ cannot exceed $hq^{r-1}$. But on the other hand, 
  $C_{\mset{H}}$ has average weight
  $n(1-1/q)=\frac{q-1-h}{q}+hq^{r-1}>hq^{r-1}$ and hence also a
  codeword of weight $>hq^{r-1}$. Contradiction.
\end{proof}
Corollary~\ref{cor:beutel} settles the determination of the maximum
size of partial line spreads ($r=2$) in $\PG(v-1,\F_q)$
completely. For $q=2$ also the maximum size of partial plane spreads
($r=3$) in $\PG(v-1,\F_2)$ is known for all $v$. The key ingredient to
this theorem is a computer construction of a partial plane spread of
size $34$ in $\PG(7,\F_2)$. The corresponding number of holes is
$2^8-1=34\cdot 7=17$, and a partial plane spread of size $35$ is
readily excluded with the aid of Proposition~\ref{prop:link}: The
associated projective binary $[10,k]$ code $C_{\mset{H}}$ would be
doubly-even by Proposition~\ref{prop:link}, but such a code does not
exist. For more details on this case and for the best currently
available general upper bounds 
we refer to \cite{smt:cost}.

\section{The Main Problem for Projective Divisible
  Codes}\label{sec:main}

In this section we formulate the general existence problem for
projective divisible codes with given parameters. In order to be as
general as possible, we note that a divisor $\Delta=p^f$ of a
$p^e$-ary code can be expressed in terms of the alphabet size $q=p^e$ as
$\Delta=q^{f/e}$. Hence, by allowing exponents $r\in\frac{1}{e}\Z^+$
we can subsume all interesting code divisors under the notion of
``$q^r$-divisibility''.\footnote{For example, the divisor $\Delta=2$ of a
  quaternary code corresponds to $r=\frac{1}{2}$.}

Let $\PD(q,r)$ be the set of all pairs of positive integers $(n,k)$,
for which a projective $q^r$-divisible linear $[n,k]$ code over $\F_q$
exists and $$\LPD(q,r)=\bigl\{n\in\Z^+;\exists k\text{ such that
}(n,k)\in\PD(q,r)\bigr\};$$ i.e., $\LPD(q,r)$ is the set of
(realizable) lengths of such codes without restricting the code dimension.
The general existence problem for projective divisible codes amounts
to the determination of the sets $\PD(q,r)$ for all prime powers $q>1$
and all $r\in\Z^+$. Since this is a formidable problem even when $q$
and $r$ are fixed to some small numbers, we try to determine the sets
$\LPD(q,r)$ first.

It turns out that each set $\LPD(q,r)$ contains all but finitely many
integers. Thus there is a well-defined function $\frobenius(q,r)$,
assigning to $q,r$ the largest integer $n$ that is not equal to the
length of a projective $q^r$-divisible linear code over $\F_q$. 
Determining $\frobenius(q,r)$ is in some sense analogous
to the well-known \emph{Frobenius Coin Problem} (see, e.g.,
\cite{brauer1942problem}), which in its simplest form asks for the
largest integer not representable as $a_1n_1+a_2n_2$ with
$a_1,a_2\geq 0$, where $n_1$ and $n_2$ are given relatively prime
positive integers. The solution is $(n_1-1)(n_2-1)-1$, as is easily
shown, and this observation together with the juxtaposition
construction for divisible codes yields an upper bound for
$\frobenius(q,r)$ (and shows that $\frobenius(q,r)$ is well-defined).
Details are contained in the next section. The determination of
$\frobenius(q,r)$ may be seen as the first important step en route to
the solution of the main problem for projective divisible codes.

\section{Constructions}\label{sec:const}

Suppose $C_i$ ($i=1,2$) are linear $[n_i,k_i]$ codes over $\F_q$ with
generating matrices $\mat{G}_i$ (in the broader sense), chosen as follows:
$\mat{G}_1$ and $\mat{G}_2$ have the same number $k$ of rows, and their
left kernels intersect only in $\{\vek{0}\}$. Then
$\mat{G}=(\mat{G}_1|\mat{G}_2)$ generates a linear $[n_1+n_2,k]$ code $C$,
called a \emph{juxtaposition} of $C_1$ and $C_2$. It is clear that $C$
is $q^r$-divisible if $C_1$ and $C_2$ are. If $C_1$ and $C_2$ are
projective, we can force $C$ to be projective as well by choosing
$\mat{G}_i$ appropriately, e.g., $\mat{G}_1=\left(
  \begin{smallmatrix}
    \mat{G}_1'\\\mat{0}
  \end{smallmatrix}\right)$, $\mat{G}_2=\left(
  \begin{smallmatrix}
    \mat{0}\\\mat{G}_2'
  \end{smallmatrix}\right)$, in which case $C$ is just the direct sum
of $C_1$ and $C_2$. This implies that the sets $\LPD(q,r)$ are
additively closed. Of course juxtaposition can be iterated, and hence
we see that in the case $\gcd(n_1,n_2)=1$ we can obtain projective
$q^r$-divisible codes of all lengths $n=a_1n_1+a_2n_2$ with
$a_1,a_2\geq 0$. Hence, choosing for $C_1$ a
$\left[\frac{q^{r+1}-1}{q-1},r+1,q^r\right]$ simplex code and for
$C_2$ a $\left[q^{r+1},r+2,q^r\right]$ first-order (generalized)
Reed-Muller code gives the bound
\begin{equation}
  \label{eq:frob}
  \begin{aligned}
    \frobenius(q,r)&\leq\frac{q^{r+1}-1}{q-1}\cdot q^{r+1}
    -\frac{q^{r+1}-1}{q-1}-q^{r+1}\\
    &=q^{2r+1}+q^{2r}+\dots+q^{r+2}-q^r-q^{r-1}-\dots-1,
  \end{aligned}
\end{equation}
as indicated in the previous section.

The implications of the juxtaposition construction for the sets
$\PD(q,r)$ are less clear, but we note the following. If
$\mset{K}_i$ denotes a set of points in $\PG(k_i-1,\F_q)$ associated
with $C_i$, $m_i$ the maximum dimension of a subspace $X_i$ with
$\mset{K}_i(X_i)=0$ (``empty subspace'') and $m=\max\{m_1,m_2\}$, then
precisely all dimensions $k_1+k_2-m\leq k\leq k_1+k_2$ can be realized by a
disjoint embedding of $\mset{K}_1$ and $\mset{K}_2$ into a common
ambient space, and hence by a projective juxtaposition of $C_1$ and
$C_2$. An example for this can be found in \cite[Th.~2]{smt:netcod16},
where a plane $\PG(2,\F_2)$ and an affine solid $\AG(3,\F_2)$ are
combined in $4$ possible ways to yield all except $1$ isomorphism type
of hole sets of partial plane spreads of size $16$ in
$\PG(6,\F_2)$. Indeed, since the affine solid in its embedding into
$\PG(3,\F_2)$ has a free $3$-subspace, the possible dimensions are
$4\leq k\leq 7$.

Viewed geometrically, the juxtaposition construction is based on
the trivial fact that the sum $\mset{K}_1+\mset{K}_2$ of two $\Delta$-divisible
multisets $\mset{K}_1$, $\mset{K}_2$ is again $\Delta$-divisible. This
observation generalizes, of course, to integral linear combinations,
shows that ($r+1$)-dimensional affine subspaces of $\PG(v-1,\F_q)$ are,
$q^r$-divisible (since $t$-subspaces with $t\geq r+1$ are)
and provides the basis for the \emph{sunflower construction}
\cite{smt:cost}. If $q$ distinct subspaces $S_1,\dots,S_q$ of
dimension at least $r+1$
$\PG(v-1,\F_q)$ pass through
a common $r$-subspace $T$ but are otherwise disjoint,
$S=(S_1\cup\dots\cup S_q)\setminus T$ is $q^r$-divisible. For the proof
note that $S=S_1+\dots+S_q-qT$ as a multiset, and that $qT$ has the divisor
$q\cdot q^{r-1}=q^r$. The construction is especially useful for $q=2$,
in which case it allows ``switching'' an $r$ subspace $T\subset S_1$
into an ($r+1$)-dimensional affine subspace $S_2\setminus
T$.\footnote{For this $S_2\supset T$ is chosen as an ($r+1$)-subspace,
  but can otherwise be arbitrary.} This increases the code length only
by one and can usually be repeated, see: 
\begin{example}
  \label{ex:q=r=2,n=19}
  According to R.~L.~Miller's database of binary doubly-even codes
  at \url{http://www.rlmiller.org/de\_codes} there exist precisely $192$
  non-equivalent binary doubly-even codes of length $19$, with all
  dimensions $3\leq k\leq 8$ realizable. However, only $3$ of these
  codes, with parameters $[19,8,4]$, $[19,7,4]$ and $[19,7,8]$, are
  projective. They correspond to the following geometric construction.

  Chose a solid $S$ in $\PG(7,\F_2)$ and $4$ planar quadrangles
  (``affine planes'')
  $A_1$, $A_2$, $A_3$, $A_4$ meeting the solid in $4$ disjoint lines $L_i$. Let
  $L$ be complement of $L_1\cup L_2\cup L_3\cup L_4$ in $S$ (which is
  also a line). Viewed as points of the quotient geometry
  $\PG(\F_2^8/S)$, the planes $A_i$ can be arranged in $3$ distinct
  ways---(i) a planar quadrangle, (ii) a line plus a plane, and (iii)
  $4$ points in
  general position. This gives $3$ inequivalent $19$-sets $A_1\cup
  A_2\cup A_3\cup A_4\cup L$ in
  $\PG(7,\F_2)$ and hence $3$ inequivalent codes with parameters
  $[19,7]$, $[19,7]$ and $[19,8]$. The point sets/codes are doubly-even, since
  they arise from $S$ by switching $L_i$ into $A_i$. The code with
  minimum distance $d=8$ corresponds to Case~(i). It can also be
  obtained by shortening the $[24,12,8]$ Golay code $\mathcal{G}_{24}$
  in $5$ (arbitrary) positions, since
  $d^\perp(\mathcal{G}_{24})=d(\mathcal{G}_{24})=8$ implies
  $d^\perp\geq 3$ for the shortened code.
\end{example}
Another important geometric construction of divisible codes introduced
in \cite{smt:cost} is the
\emph{cone construction}, which increases the divisor from $q^r$ to
$q^{r+1}$ (or, in its most general form using an $s$-dimensional
vertex, to $q^{r+s}$). Let $H$ be a
hyperplane of $\PG(v-1,\F_q)$. A cone $K$ with
vertex $P\notin H$ and base $B\subseteq H$ is defined as the union of
the lines $PQ$ with $Q\in B$. If $B$ is $q^r$-divisible then 
the number of points of $K$ outside any hyperplane through $P$ is
clearly a multiple of $q^{r+1}$, and we may adjust the multiplicity of
$P$ in $K$ without affecting this property. Since the number of points
of $K\setminus\{P\}$ outside every other hyperplane is $(q-1)\card{B}$, it
follows that $K\setminus\{P\}$ is $q^{r+1}$-divisible if
$\card{B}\equiv 0\pmod{q^{r+1}}$, and $K$ is $q^{r+1}$-divisible if
$\card{B}(q-1)\equiv -1\pmod{q^{r+1}}$.
\begin{example}
  \label{ex:cone}
  A projective basis of $\PG(k-1,\F_2)$ corresponds to the binary
  $[k+1,k,2]$ even-weight code and gives via the cone construction a
  self-dual doubly-even $[2k+2,k+1,4]$ code if $k\equiv3\pmod{4}$ and a
  doubly-even $[2k+3,k+1,4]$ code if $k\equiv2\pmod{4}$. Generating
  matrices for $k=6,7$ are as follows:
  \begin{equation*}
    \setlength{\arraycolsep}{1pt}
    \renewcommand{\arraystretch}{.75}
    \left(
      \begin{array}{cccccccccccccc|c}
        1&1&1&1&0&0&0&0&0&0&0&0&0&0&0\\
        1&1&0&0&1&1&0&0&0&0&0&0&0&0&0\\
        1&1&0&0&0&0&1&1&0&0&0&0&0&0&0\\
        1&1&0&0&0&0&0&0&1&1&0&0&0&0&0\\
        1&1&0&0&0&0&0&0&0&0&1&1&0&0&0\\
        1&1&0&0&0&0&0&0&0&0&0&0&1&1&0\\\hline
        1&0&1&0&1&0&1&0&1&0&1&0&1&0&1\\
      \end{array}
    \right),\quad
    \left(
      \begin{array}{cccccccccccccccc}
        1&1&1&1&0&0&0&0&0&0&0&0&0&0&0&0\\
        1&1&0&0&1&1&0&0&0&0&0&0&0&0&0&0\\
        1&1&0&0&0&0&1&1&0&0&0&0&0&0&0&0\\
        1&1&0&0&0&0&0&0&1&1&0&0&0&0&0&0\\
        1&1&0&0&0&0&0&0&0&0&1&1&0&0&0&0\\
        1&1&0&0&0&0&0&0&0&0&0&0&1&1&0&0\\
        1&1&0&0&0&0&0&0&0&0&0&0&0&0&1&1\\\hline
        1&0&1&0&1&0&1&0&1&0&1&0&1&0&1&0\\
      \end{array}
    \right).
  \end{equation*}
  Here $v=k+1$, $H$ is the hyperplane with equation $x_{k+1}=0$, and
  $P=(0:0:\dots:1)$. The reader should recognize the second matrix as
  one of the basic self-dual code constructions for $q=2$. The first
  matrix corresponds to the $5$th isomorphism type of hole sets of
  partial plane spreads of size $16$ in $\PG(6,\F_2)$;
  cf. \cite[Th.~2]{smt:netcod16}.
\end{example}
Several other constructions are known---for example concatenating a
$q$-divisible code over $\F_{q^r}$ with an $r$-dimensional simplex
code over $\F_q$ obviously yields a $q^r$-divisible code---and a
wealth of further examples:
Higher-order (generalized) Reed-Muller codes are divisible by Ax's
Theorem \cite{ax64}, semisimple abelian group algebra codes under
certain conditions by Delsarte-McEliece \cite{delsarte-mceliece} (for
these two theorems see also \cite{ward2001divisible_survey}),
and projective two-weight codes if the weights satisfy
$w_2>w_1+1$ \cite{delsarte1973algebraic}.\footnote{This condition is
  always satisfied if $k\geq 3$ and the code is not a punctured
  simplex code \cite[Prop.~2]{bouyukliev2006projective}.}
For the latter the survey \cite{calderbank1986geometry} is a
particularly useful source.

\section{Results for $q=2$}\label{sec:q=2}

First we determine the length-dimension pairs realizable by a binary
projective $2$-divisible code. The case $r=1$ is the only case, where we can
determine the set $\PD(2,r)$ completely.

\begin{theorem}
  \label{thm:q=2,r=1}
  The set $\PD(2,1)$ consists of all pairs $(n,k)$ of positive
  integers satisfying $k+1\leq n\leq 2^k-1$ and $n\notin\{2^k-3,2^k-2\}$.
\end{theorem}

\begin{proof}
  It is clear that the stated conditions are necessary for the
  existence of a projective $2$-divisible $[n,k]$ code.

  For the converse we consider $k$ as fixed and use induction on $n$
  in the range $k+1\leq n\leq 2^{k-1}$. The $[k+1,k]$
  even-weight code, which corresponds to a projective basis of
  $\PG(k-1,\F_2)$, provides the base for the induction. Now assume
  that $\mset{K}$ is a $2$-divisible spanning point set in
  $\PG(k-1,\F_2)$ with $k+1\leq n=\card{\mset{K}}<2^{k-1}$. If $\mset{K}$
  has a tangent $L$, we can switch the point of tangency into the
  other two points on $L$ and increase $n$ by one.\footnote{The new
    point set will of course be spanning as well.} If $\mset{K}$ has
  no tangent then the complementary point set
  $\points\setminus\mset{K}$ must be a subspace (since it is closed 
  with respect to taking the join of any two of its points). 
  This can only occur for $n\geq 2^{k-1}$.

  Since the complement of a $2$-divisible point set in
  $\PG(k-1,\F_2)$ is $2$-divisible, we get $(n,k)\in\PD(2,1)$ also for 
  $2^{k-1}<n\leq 2^k-k-2$. The proof is concluded by removing from
  $\points$ a projective basis in an $l$-subspace, $2\leq l\leq k-1$, 
  which is $2$-divisible. This covers the range $2^k-k-1\leq n\leq
  2^k-4$ and completes the proof.
\end{proof}

Now we assume $r\geq 2$ and restrict attention to the sets
$\LPD(2,r)$. First we sharpen the simple upper bound \eqref{eq:frob},
which for $q=2$ is $\frobenius(2,r)\leq 2^{2r+2}-3\cdot 2^{r+1}+1$.

\begin{theorem}
  \label{thm:frob2}
  For $k\geq 2$ we have $\frobenius(2,r)\leq 2^{2r}-2^{r-1}-1$.
\end{theorem}

The proof uses a combination of the switching and concatenation
constructions described in Section~\ref{sec:const} together with the
observation that $n\in\LPD(2,r)$ implies
$n+(2^{r+1}-1)\Z\subseteq\LPD(2,r)$ (juxtaposition with
($r+1$)-dimensional simplex codes).

Theorem~\ref{thm:frob2} is sharp for $r=2$, i.e.,
$\frobenius(2,2)=13$. In fact it is not
difficult to see that a projective doubly-even binary code of length
$n$ does not exist for
$n\leq 6$ and $9\leq n\leq 13$, and hence
$\LPD(2,2)=\{7,8\}\cup\Z_{\geq 14}$.

The case $r=3$ (``triply-even'' codes) was settled in \cite{smt:cost}
with one exception: $\frobenius(2,3)\in\{58,59\}$, and $\LPD(2,3)$
contains
$\{15,16,30,31,32,45,46,\linebreak[2] 47,48,49,50,51\}\cup\Z_{\geq
  60}$ and possibly $59$. The non-existence proof in the remaining
cases uses the methods developed in \cite{kurz17} and adhoc linear
programming bounds derived from the first 
four MacWilliams
identities.

The existence of a projective triply-even binary code of length
$59$ remains an open question. If such a code exists it must be
constructible from two projective doubly-even codes of lengths $27$
and $32$ using the juxtaposition construction in
\cite[Prop.~19]{betsumiya2012triply}.\footnote{
The putative code contains a codeword of weight $32$; hence
  \cite[Prop.~22]{betsumiya2012triply} applies.}

Now we are going to give a classification of short projective $2^r$-divisible binary codes for $r \leq 3$.
The case $r=1$ is special as the set of all even-weight words forms a linear subspace of $\F_2^n$, the $[n,n-1]$ \emph{even-weight code}.
Thus, we can produce all types of projective $2$-divisible $[n,k]$ codes by starting with the even-weight code and recursively enumerating the codes $C$ of codimension~1, as long as $C$ is projective and not isomorphic to some previously produced code.
While this somewhat simplistic approach could certainly be improved in various ways, it is good enough to produce the results shown in Table~\ref{tbl:Delta2}.

\begin{table}[htp]
\caption{Classification of projective $2$-divisible binary codes}
\label{tbl:Delta2}
\centering
        $\begin{array}{r||r|rrrrrrrrrrrrr}
        \multicolumn{1}{c||}{n} & \multicolumn{1}{c|}{\Sigma} & k= & 2 & 3 & 4 & 5 & 6 & 7
& 8 & 9 & 10 & 11 & 12 & 13 \\
        \hline\hline
        3 & 1 & & 1 \\
        4 & 1 & & & 1 \\
        5 & 1 & & & & 1 \\
        6 & 2 & & & & 1 & 1 \\
        7 & 4 & & & 1 & 1 & 1 & 1 \\
        8 & 7 & & & & 2 & 2 & 2 & 1 \\
        9 & 12 & & & & 1 & 4 & 4 & 2 & 1 \\
        10 & 26 & & & & 1 & 6 & 9 & 6 & 3 & 1 \\
        11 & 61 & & & & 1 & 8 & 21 & 18 & 9 & 3 & 1 \\
        12 & 169 & & & & 1 & 11 & 45 & 59 & 35 & 13 & 4 & 1 \\
        13 & 505 & & & & & 12 & 91 & 182 & 141 & 57 & 17 & 4 & 1 \\
        14 & 1944 & & & & & 12 & 191 & 633 & 668 & 318 & 94 & 22 & 5 & 1 
        \end{array}$
\end{table}

The projective binary doubly-even $[n,k]$ codes with lengths $n\leq 26$
($n=26$, $k=12$ not yet finished) have been classified by using the
command \begin{center}\texttt{sage.coding.databases.self\_orthogonal\_binary\_codes()}\end{center}
in SageMath \cite{sagemath}. The result is shown in
Table~\ref{tbl:Delta4}.

 \begin{table}[htp]
 \caption{Classification of projective $4$-divisible binary codes}
 \label{tbl:Delta4}
 \centering
 	$\begin{array}{r||r|rrrrrrrrrrr}
 	\multicolumn{1}{c||}{n} & \multicolumn{1}{c|}{\Sigma} & k= & 3 & 4 & 5 & 6 & 7 & 8 & 9 & 10 & 11 & 12 \\
 	\hline\hline
 	7 & 1 & & 1 \\
 	8 & 1 & & & 1 \\
 	\hline
 	14 & 1 & & & & & 1 \\
 	15 & 4 & & & 1 & 1 & 1 & 2 \\
 	16 & 9 & & & & 2 & 2 & 3 & 2 \\
 	17 & 3 & & & & & 1 & 1 & 1 \\
 	18 & 3 & & & & & 1 & 1 & 1 \\
 	19 & 3 & & & & & & 2 & 1 \\
 	20 & 7 & & & & & & 2 & 4 & 1 \\
 	21 & 24 & & & & & 2 & 7 & 9 & 6 \\
 	22 & 101 & & & & & 3 & 24 & 41 & 24 & 9 \\
 	23 & 503 & & & & 1 & 11 & 83 & 201 & 146 & 50 & 11 \\
 	24 & 1856 & & & & 1 & 15 & 181 & 679 & 663 & 250 & 58 & 9 \\
 	25 & 4972 & & & & & 6 & 234 & 1688 & 2162 & 748 & 121 & 13 \\
 	26 & \mathit{\geq 21843} & & & & & 3 & 376 & 6021 & 11010 & 3920 & 478 & \mathit{\geq 35}
 	\end{array}$
 \end{table}
We note that self-dual doubly-even codes are necessarily projective,
and hence the classification of such codes for a particular $n$ yields
the classification of projective doubly-even $[n,n/2]$ codes. For
example, from \cite{conway-pless-sloane92} we know that there are
exactly $85$ types of such codes for $n=32$.

In \cite{betsumiya2012triply}, the binary $8$-divisible codes of
length $48$ have been classified.  On the first author's web page
\url{http://www.st.hirosaki-u.ac.jp/~betsumi/triply-even/}, all $7647$
types of \emph{self-complementary} (i.e., containing the all-one word)
binary $8$-divisible codes are given explicitly.  From this data, we
have derived the classification of all projective binary $8$-divisible
codes of length up to $48$.  First, the self-complementary ones of
length exactly $48$ are produced by simply going through the list of
all $7647$ codes and checking them for projectivity, which leads to $291$
types of codes.

For all other codes, we note that lengthening to $n=48$ (padding
codewords with zeros) and then augmenting by the all-one word of
length $48$, a binary self-complementary (not necessarily projective)
$8$-divisible code is produced.  Therefore we can produce all codes
by going through the list of $7647$ codes $C$, enumerating all
codimension~1 subcodes $C'$ of $C$ not containing the all-one word
(their number is $2^{\dim(C) - 1}$), removing all-zero coordinates,
and checking the resulting code for projectivity. No code is lost in
this way, but it may happen that the same isomorphism type of a code
is produced several times. Filtering the list of codes
for isomorphic copies produced the result shown in Table~\ref{tbl:Delta8}.

\begin{table}[htp]
\caption{Classification of projective $8$-divisible binary codes}
\label{tbl:Delta8}
\centering
$\begin{array}{r||r|rrrrrrrrrrrrr}
	\multicolumn{1}{c||}{n} & \multicolumn{1}{c|}{\Sigma} & k= & 4 & 5 & 6 & 7 & 8 & 9 & 10 & 11 & 12 & 13 & 14 & 15 \\
	\hline\hline
	15 & 1 & & 1 \\
	16 & 1 & & & 1\\
	\hline
	30 & 1 & & & & & & 1 \\
	31 & 6 & & & 1 & 1 & 1 & 2 & 1 \\
	32 & 11 & & & & 2 & 2 & 3 & 3 & 1 \\
	\hline
	45 & 6 & & & & & & 2 & 1 & 1 & 1 & 1 \\
	46 & 51 & & & & & & 6 & 18 & 14 & 8 & 4 & 1\\
	47 & 856 & & & & 1 & 11 & 100 & 299 & 274 & 122 & 40 & 8 & 1 \\
	48 & 2973 & & & & 1 & 15 & 211 & 921 & 1071 & 529 & 173 & 44 & 7 & 1 
\end{array}$
\end{table}

We have the following constructions for projective $2^r$-divisible
binary $[n,k]$ codes. The codes are described in terms of their
associated point sets in $\PG(k-1,\F_2)$.

\begin{itemize}
	\item $n=2^{r+1}-1$: A projective $r$-flat ($[2^{r+1}-1,r+1]$
          simplex code)
	\item $n=2^{r+1}$: An affine $(r+1)$-flat ($[2^{r+1}, r+2]$
          1st-order RM code) 
	\item $n=2^{r+2}-2$: The (unique) disjoint union of two
          projective $r$-flats, of ambient space dimension $k=2r+2$
	\item $n=2^{r+2}-1$: The disjoint union of a projective
          $r$-flat $F$ and an affine $(r+1)$-flat $A$. We get one type
          of code for each intersection dimension
          $s\in\{0,\ldots,r+1\}$ of $F$ with the hyperplane at
          infinity of $A$.  The ambient dimension is $k = 2r + 3 - s$.
          In the case $s = r+1$, we simply get a projective
          ($r+1$)-flat, which is even $2^{r+1}$-divisible.

          A further code is given by the set of $7$ projective $(r-1)$-flats passing
          through a common $(r-2)$-flat $V$ such that the image modulo
          $V$ is a projective basis. The ambient dimension is
          $k=r+5$.
	\item $n=2^{r+2}$: The disjoint union of two affine
          $(r+1)$-flats $A_1$ and $A_2$.  There are two types of such
          unions for each $k \in \{r+3,\ldots,2r+3\}$ and a single
          type for $k = 2r+4$.  One of the types for $k=r+3$ actually
          is an affine $(r+2)$-flat, which is even
          $2^{r+1}$-divisible.

          There are two more types: Let $\{S_1,\ldots,S_8\}$ be a set
          of $8$ projective $(r-1)$-flats passing through a common
          $(r-2)$-flat $V$, such that the image modulo $V$ is a
          projective basis. Then
          $(S_1\cup \ldots\cup S_8) \setminus V$ yields a suitable
          code with $k = r+6$.

          Furthermore, let $X$ be the disjoint union of maximum
          possible dimension of $\PG(1,\F_2)$ and a projective
          basis of $\PG(3,\F_2)$. Then $\card{X} = 8$ and
          $\dim\langle X\rangle = 6$.  Now let $V$ be a projective
          $(r-2)$-flat disjoint from $\langle X\rangle$.  Then
          $(\bigcup_{P\in X}\langle P, V\rangle ) \setminus V$ yields
          a suitable code with $k = r+5$.
\end{itemize}
Note that the three constructions involving an ($r-1$)-subspace $V$
(``vertex'') are examples of the generalized cone construction (with a
vertex of dimension $s=r-1$) mentioned in Section~\ref{sec:const}.

For $r\in\{1,2,3\}$, the above constructions cover all types of codes
of the corresponding lengths, with the exception of $r=1$, $n=6$, where additionally the even-weight code shows up.

For $n=3(2^{r+1}-1) = 2^{r+2} + 2^{r+1} - 3$, suitable codes can be produced as the disjoint union of three projective $r$-flats.
This yields a unique type of code for each ambient space dimension $k\in\{2r+2,\ldots,3r+3\}$.
In the case $k=2r+2$, the resulting code is a two-weight code with weights $2^{r+1}$ and $2^{r+1} + 2^r$.
However, for all $r\in\{1,2,3\}$, there are projective $2^r$-divisible codes different from this construction.
The most interesting case is $r=3$, $n=45$, where only a single further code shows up.
It is another $[45,8]$ two-weight code with weights $16$ and $24$; see
\cite[Th.~4.1]{haemers-peeters-rijckevorsel17-1}. The associated point
set $\mset{K}$ in $\PG(7,\F_2)$ consists of a projective basis
$P_1,\dots,P_9$ and the $\binom{9}{2}=36$ remaining points on the
lines $P_iP_j$.
Furthermore, it is worth mentioning that also in the case $r=2$, $n=21$ there is a second $[21,6]$ two-weight with weights $8$ and $12$, see~\cite{bouyukliev2006projective}.

For further information on $r=2$, $n=15$ see \cite{smt:netcod16} and on $r=2$, $n=17$ see \cite[Sect.~1.6.1]{smt:cost}.
A further settled case worth mentioning is $r=3$, $n=51$, see \cite[Lem.~24]{smt:cost}.
In that case, there is a unique code, which can be constructed as the concatenation of an ovoid in $\PG(3,\F_4)$ with the binary $[3,2]$ simplex code.

\section*{Acknowledgement}

The authors are grateful to Robert L.~Miller for explaining how to use
his database of doubly-even codes in SageMath, and to Koichi Betsumiya
and Akihiro Munemasa for additional information on the classification
of triply-even codes. Thomas Honold was supported by the National Natural
Science Foundation
of China under Grant 61571006.
The authors would like to acknowledge the financial support provided by COST --
\emph{European Cooperation in Science and Technology}.
The first, fourth, and fifth author were supported in part by the grant KU 2430/3-1 and WA 1666/9-1
-- Integer Linear Programming Models for Subspace Codes and Finite 
Geometry -- from the German Research Foundation.
 

\begin{thebibliography}{10}
\providecommand{\url}[1]{\texttt{#1}}
\providecommand{\urlprefix}{URL }

\bibitem{ax64}
Ax, J.: Zeroes of polynomials over finite fields.
Amer. J. Math.
86,  255--261 (1964)

\bibitem{betsumiya2012triply}
Betsumiya, K., Munemasa, A.: On triply even binary codes.
J. Lond. Math. Soc. $\mathrm{(2)}$
86(1),  1--16 (2012)

\bibitem{bouyukliev2006projective}
Bouyukliev, I., Fack, V., Willems, W., Winne, J.: Projective two-weight codes
  with small parameters and their corresponding graphs.
  Des. Codes Cryptogr.
  41(1),  59--78 (2006)

\bibitem{brauer1942problem}
Brauer, A.: On a problem of partitions.
Amer. J. Math.
64(1),
   299--312 (1942)

\bibitem{calderbank1986geometry}
Calderbank, R., Kantor, W.: The geometry of two-weight codes.
Bull. London Math. Soc.
18(2),  97--122 (1986)

\bibitem{delsarte1973algebraic}
Delsarte, P.: An algebraic approach to the association schemes of coding
  theory. Philips research reports supplements (10),  103 (1973)

\bibitem{conway-pless-sloane92}
Conway, J.H., Pless, V., Sloane, N.J.A.: The binary self-dual codes of length
  up to $32$: A revised enumeration.
 J. Combin. Theory Ser. A 
60,  183--195 (1992)

\bibitem{delsarte-mceliece}
Delsarte, P., McEliece, R.J.: Zeroes of functions in finite abelian group
  algebras.
Amer. J. Math.
  98,  197--224 (1976)

\bibitem{dodunekov1998codes}
Dodunekov, S., Simonis, J.: Codes and projective multisets.
Electron. J. Combin.
5(R37),  1--23 (1998)

\bibitem{smt:netcod16}
Honold, T., Kiermaier, M., Kurz, S.: Classification of large partial plane
  spreads in {$\PG(6,\F_2)$} and related combinatorial objects (Apr 2016),
  submitted for publication. Preprint arXiv:1606.07655

\bibitem{smt:cost}
Honold, T., Kiermaier, M., Kurz, S.: Partial spreads and vector space
  partitions (Nov 2016), preprint arXiv:1611.06328

\bibitem{haemers-peeters-rijckevorsel17-1}
Haemers, W.H., Peeters, R., van Rijckevorsel, J.M.: Binary codes of strongly regular graphs.
  Des. Codes Cryptogr.
17, 187--209 (1999)

\bibitem{huffman-pless03}
Huffman, W.C., Pless, V.: Fundamentals of Error-Correcting Codes. Cambridge
  University Press (2003)

\bibitem{koetter-kschischang08}
Koetter, R., Kschischang, F.: Coding for errors and erasures in random network
  coding.
   IEEE Trans. Inform. Theory 
  54(8),  3579--3591 (Aug 2008)

\bibitem{kurz17}
Kurz, S.: {Packing vector spaces into vector spaces}.
Australas. J. Combin.
68(1) (2017), to appear

\bibitem{silva-kschischang-koetter08}
Silva, D., Kschischang, F., Koetter, R.: A rank-metric approach to error
  control in random network coding.
   IEEE Trans. Inform. Theory 
  54(9),  3951--3967 (Sep 2008)

\bibitem{silva-kschischang-koetter10}
Silva, D., Kschischang, F., Koetter, R.: Communication over finite-field matrix
  channels.
   IEEE Trans. Inform. Theory 
  56(3),  1296--1306 (Mar
  2010)

\bibitem{sagemath}
Stein, W.A., et al.: Sage Mathematics Software (Version
  7.1), The Sage Development Team, 2016,
  \url{http://www.sagemath.org}

\bibitem{tsfasman-vladut95}
Tsfasman, M.A., Vl{\u a}du{\c t}, S.G.: Geometric approach to higher weights.
   IEEE Trans. Inform. Theory 
  41,  1564--1588 (1995)

\bibitem{ward2001divisible_survey}
Ward, H.: Divisible codes---a survey.
Serdica Math. J.
27(4), 263--278 (2001)

\end{thebibliography}

\def\cprime{$'$}

\end{document}